\documentclass[10pt]{article}
\usepackage{a4wide}
\usepackage{latexsym}
\usepackage{amsmath}
\usepackage{amsfonts}
\usepackage{amssymb}
\usepackage{amsthm}
\usepackage[all]{xy}

\def\CC{{\mathbb{C}}}
\def\NN{{\mathbb{N}}}

\def\cH{\mathcal{H}}
\def\cL{\mathcal{L}}
\def\cR{\mathcal{R}}

\def\cK{\mathcal{K}}

\def\cB{\mathcal{B}}

\def\dd{^{\| d}}
\def\de{^{\| e}}

\def\io{^{-1}}

\def\grp{^\#}

\def\iff{\Leftrightarrow}

\def\beqn{\begin{eqnarray*}}
\def\eeqn{\end{eqnarray*}}
\def\barr{\begin{array}}
\def\earr{\end{array}}

\def\ben{\begin{enumerate}}
\def\een{\end{enumerate}}
\def\l({\left(}
\def\r){\right)}

\def\bmx{\left[\begin{array}}
\def\emx{\end{array}\right]}

\newcommand{\fonction}[5]{$$\begin{array}[t]{lccl}
#1: & #2 & \longrightarrow & #3 \\
    & #4 & \longmapsto & #5 \end{array}$$}

\newtheorem{theorem}{Theorem}[section]
\newtheorem{lemma}[theorem]{Lemma}
\newtheorem{proposition}[theorem]{Proposition}
\newtheorem{corollary}[theorem]{Corollary}
\newtheorem{definition}[theorem]{Definition}

\begin{document}

\title{Natural generalized inverse and core of an element in semigroups, rings and Banach and Operator algebras}

\author{Xavier Mary\footnote{email: xavier.mary@u-paris10.fr}\\
\textit{\small Universit\'e Paris-Ouest Nanterre-La D\'efense, Laboratoire Modal'X}}

\date{}

\maketitle

%\subjclass{Primary 15A09; 20M99} \keywords{generalized inverse,
% Green's relations, semigroup}

\begin{abstract}
Using the recent notion of inverse along an element in a semigroup, and the natural partial order on idempotents, we study bicommuting generalized inverses and define a new inverse called natural inverse, that generalizes the Drazin inverse in a semigroup, but also the Koliha-Drazin inverse in a ring. In this setting we get a core decomposition similar to the nilpotent, Kato or Mbekhta decompositions. In Banach and Operator algebras, we show that the study of the spectrum is not sufficient, and use ideas from local spectral theory to study this new inverse. \end{abstract}

\section{Introduction}

%In this paper, $S$ denotes a semigroup, $S^1=S\cup\{1\}$ is the semigroup obtained by $S$ together with an identity in the usual way, and $R$ is a ring (with identity).

In this paper, $S$ is a semigroup and $S^1$ denotes the monoid generated by $S$. $E(S)$ denotes the set of idempotents. For any subset $A$ of $S$, $A'=\{x\in S,\, xa=ax\; \forall a\in A\}$ denotes the commutant (or centralizer) of $A$.

We say $a$ is (von Neumann) regular in $S$ if $a\in aSa$. A particular solution to $axa=a$ is called an associate, or inner inverse, of $a$.  A solution to $xax=a$ is called a weak (or outer) inverse. Finally, an element that satisfies $axa=a$ and $xax=x$ is called an inverse (or reflexive inverse, or relative inverse) of $a$ and is denoted by $a'$. The set of all associates of $a$ is denoted by $A(a)$, and the set of weak inverses of $a$ by $W(a)$. %If $a'\in V(a)$, we also say that $(a, a')$ is a regular pair.

A commuting inverse, if it exists, is unique and denoted by $a\grp$. It is usually called the group inverse of $a$.\\

We will make use of the Green's preorders and relations in a semigroup \cite{Green51}. For elements $a$ and $b$ of $S$, Green's preorders  $\leq_{\cL}$, $\leq_{\cR}$ and
$\leq_{\cH}$ are defined by
\begin{align*}
a \leq_{\cL} b&\Longleftrightarrow S^1 a\subset S^1 b\Longleftrightarrow \exists x\in S^1,\;  a = xb;\\ %\label{eq1}\\
a \leq_{\cR} b& \Longleftrightarrow  aS^1\subset  bS^1\Longleftrightarrow \exists x\in S^1,\;  a = bx;\\ %\label{eq2}\\
a \leq_{\cH} b&\Longleftrightarrow \{a\leq_{\cL} b\text{ and }a\leq_{\cR} b\}.
\end{align*}

If $\leq_{\cK}$ is one of these preorders, then $a\cK b\iff \{a\leq_{\cK}b \text{ and } b\leq_{\cK}a\}$, and $\cK_a=\{b\in S,\; b\cK a\}$ denotes the $\cK$-class of $a$.

We recall the following characterization of group invertibility in terms of Green's relation $\cH$ (see \cite{Green51}, \cite{Clifford56}):
\begin{lemma}\label{lemma_Clifford}
$a\grp$ exists if and only if $a\cH a^2$ if and only if $\cH_a$ is a group.
%\item Let $(a,a')$ be a regular pair. Then $aa'=a'a$ if and only if $a\cH a'$. 
\end{lemma}

The study of generalized inverses has been conducted in many different mathematical areas, from semigroup theory to Operator theory, and applied to various domains such has Markov chains or differential equations. In these studies, it may be useful to consider commuting (or bicommuting) inverses. Since the existence of a commuting inner inverse is a very strong property, it is common to look at outer commuting inverses, following the seminal work of M. Drazin \cite{Drazin58}, who introduced the Drazin inverse in the context of semigroups and rings. Later, this inverse has been generalized in the setting of operators by Koliha \cite{Koliha96} using spectral properties and functional calculus. This generalized Drazin inverse (also called Koliha-Drazin inverse) finds many applications, in particular to singular differential equations.\\

In \cite{mary} the author introduced a special outer inverse, called inverse along an element in the context of semigroups. The aim of this article is to use this new inverse to we study bicommuting generalized inverses. Then, using the natural partial order on idempotents, we will define a new inverse called natural inverse, that generalizes the Drazin inverse in a semigroup, but also the Koliha-Drazin inverse in a ring. In this setting, this provides a decomposition of an element similar to the nilpotent, Kato or Mbekhta decompositions \cite{Mbekhta87}. In the first sections we introduce the main notions (inverse along an element, natural generalized inverse) entirely in the semigroup setting. We then study further properties of the natural inverse in rings, making the link with quasipolar (generalized Drazin invertible) elements (\cite{Harte91}, \cite{Harte09}, \cite{Koliha96}, \cite{Koliha02}). In the last sections, a particular attention is given to Banach and operators algebras. The main result is that this inverses relies on finer properties than spectral properties only. Local spectral theory\cite{Harte09local} is then an interesting tool. 
 
\section{Inverse along an element}

\subsection{Definition and first properties}
The inverse along an element was introduced in \cite{mary}, and in \cite{maryPatricio}, it was interpreted as a kind of inverse modulo $\cH$. We recall the definition and properties of this inverse.

\begin{definition}
Given $a,d$ in $S$, we say $a$ is invertible along $d$ if there exists $b\in S$ such that  $bad=d=dab$ and $b\leq_\cH d$. If such an element exists then it is unique and is denoted by $a^{\parallel d}$.
\end{definition}

An other characterization is the following:
\begin{lemma}
$a$ is invertible along $d$ if and only if there exists $b\in S$ such that $bab=b$ and $b\cH d$, and in this case $a^{\parallel d}=b$.
\end{lemma}

\begin{theorem}\label{thexist}
Let $a,d\in S$. Then the following are equivalent:
\ben
\item $a^{\parallel d}$ exists.
\item $d\leq_{\cR} da$ and $(da)\grp$ exists.
\item $d\leq_{\cL} ad$ and $(ad)\grp$ exists.
\item $dad\cH d$.
\item $d\leq_{\cH} dad$.
\een
In this case, $$b=d(ad)^{\sharp}=(da)^{\sharp}d.$$
\end{theorem}

For an other look at this inverse, we also refer to \cite{Drazin2012}, where M. Drazin independently defined an new outer inverse that is actually similar to the inverse along an element.

\subsection{Commutativity and idempotents}

A remarkable feature of the inverse along an element is the following (theorem 10 in \cite{mary}).

\begin{theorem}\label{commut}
Let $a,d\in S$ and pose $A=(a,d)$. If $a$ is invertible along $d$,
then $(a)^{\parallel d}\in A''$.
\end{theorem}

As a direct corollary, we get:
\begin{corollary}\label{corcommut}
Let $a,d\in S$, $dad\cH d$ and pose $b=a\dd$. If $ad=da$, then
$ab=ba$ and $bd=db$. 
\end{corollary}

%It follows that the inverse of $a$ along $d\in \{a\}''$, if it exists, is also in $\{a\}''$. As a consequence, and since $a\in \{a\}''$, we can work in the commutative subsemigroup $\{a\}''$ to find outer inverses in the bicommutant.

%\begin{lemma}   
%\begin{enumerate}
%\item Suppose $ad=da$. Then $a\in S$ is invertible along $d\in S$ if and only if $ad \cH d^2\cH d$.  In this case, $a\dd a=a a\dd$.
%\item Suppose $bab=b$, $ba=ab$. Then $b=a^{\parallel e}$, with $e=ab\in \{a\}'$.
%\end{enumerate}
%\end{lemma}

%The proof follows from the characterization of invertibility in terms of group inverses. $a^{\parallel d}$ exists $\iff$ $ad\cL d$ and $\cH_{ad}$ is a group  $\iff$ $da\cR d$ and $\cH_{da}$ is a group $\iff$ $\cH_d=\cH_{ad}$ is a group since $ad=da$.

We define the following sets:
\ben
\item $\Sigma_0(a)=\{e\in E(S),\; eae\cH e\}$;
\item $\Sigma_1(a)=\{a\}'\cap \Sigma_0(a)$;
\item $\Sigma_2(a)=\{a\}''\cap \Sigma_0(a)$.
\een

(If $S$ is commutative, or the idempotents are central, then the three sets are equal. We then simply denote it $\Sigma(a)$.)

\begin{lemma}
Let $e\in E(S)$ and $a\in S$ such that $ae=ea$. Then $e\in \Sigma_0(a)\iff e\leq_{\cH} a$.
\end{lemma}

\proof
Assume $e\in \Sigma_0(a)$. Then $e\leq_{\cH} eae=ea=ae\leq_{\cH} a$. Conversely, if $e\leq_{\cH} a$ and $ae=ea$, then $e\leq_{\cR} a\Rightarrow e=ee \leq_{\cR} ea \leq_{\cR} e$ that is $e\cR ea$. But $ea=ae$, hence $e\cR ea\Rightarrow e\cR ae\Rightarrow e=ee\cR eae$. By symmetry, we get $e\cH eae$. 
\endproof
 
Combining the previous lemmas and theorems we get:
\begin{theorem}
\fonction{\tau_a}{W(a)}{E(S)}{x}{ax} 
\begin{itemize}
\item is one to one from $W(a)\cap \{a\}'$ onto $\Sigma_1(a)$;
\item is one to one from $W(a)\cap \{a\}''$ onto $\Sigma_2(a)$.
\end{itemize}  
Its reciprocal $\tau_a^{-1}$ associates $e$ to $b=a^{\parallel e}$.
\end{theorem}

\proof
Let $b,c\in W(a)\cap \{a\}'$. Then $ab=ac\Rightarrow b=bab=bac$. But also $ba=ca$ by commutativity and $bac=cac=c$. Finally $b=c$. Obviously, $ab=ba=e$ is an idempotent commuting with $a$. Conversely, if $e\in \Sigma_1(a)$, then $e\leq_{\cR} a\Rightarrow e=ee \leq_{\cR} ea=ae\leq_{\cR} e$. Also $e\leq_{\cL} a\Rightarrow e=ee \leq_{\cR} ae=ea\leq_{\cL} e$. It follows that $ea=ea\cH e$, $a$ is invertible along $e$. Pose $b=a^{\parallel e}$. Then $b\in \{a,e\}''$ hence $ab=ba$ and $ab=abe=bae=e$.\\
For the second statement, we have only to prove that $\tau_a$ maps $W(a)\cap \{a\}''$ onto $\Sigma_2(a)$, but this follows from
theorem \ref{commut}. 
\endproof

As a consequence, looking for commuting or bicommuting outer inverses can be handled through idempotents.\\

Recall that any set of idempotents may be partially ordered by $e\leq f \iff ef=fe=e$, the natural partial order, and if this set is commutative, then this partial order is compatible with multiplication.
%$$t\leq s \iff \exists x,y\in T^1,\; xs=xt=t=ty=ts.$$
%and that this partial order reduces to 
We then have two partial orders on $E(S)$, the natural partial order and the $\cH$ preorder (that reduces to a partial order for idempotents since a $\cH$-class contains at most one idempotent \cite{Green51}). Actually, they coincide for idempotents. If $e\leq f$, then $e=ef=fe$ and $e\leq_{\cH} f$ and conversely, if $e=fx=yf$ then $fe=ffx=fx=e=yf=yff=ef$.\\

It is interesting to notice that even in the noncommutative case, invertibilty along an idempotent $e$ can be expressed as invertibity in the corner semigroup $eSe$.
\begin{lemma}
Let $a\in S, \; e\in E(S)$. Then $e\in \Sigma_0(a)$ ($a^{\parallel e}$ exists) if and only if $eae$ is invertible in the corner monoid $eSe$.
In this case
 $$a^{\parallel e}=(ea)\grp e=e (ae)\grp=(eae)\grp=(eae)^{-1}.$$
\end{lemma}

\proof
Assume $a\de$ exists. Then $a\de \cH e$ hence $a\de=e a\de=a\de e=e a\de e\in eSe$. It also satisfies $a\de a e=e = e a a\de$ hence $a\de (eae)=e=(eae)a\de$ and $eae$ is invertible in the monoid $eSe$ (with unit $e$).\\
Conversely, assume $eae$ is invertible in $eSe$ with inverse $b\in eSe$. Then $b\leq_{\cH} e$ and $bae=b(eae)=e=(eae)b=eab$ and $b$ is the inverse of $a$ along $e$.
\endproof

Finally, note that $\Sigma_2(a)$ is a commutative band (commutative semigroup of idempotents, semillatice with $e\vee f=ef=fe$).

\begin{proposition}
$\Sigma_2(a)$ is a commutative subsemigroup of $S$.
\end{proposition}

\proof
If $e,f\in \Sigma_2(a)$, then $ef=fe\leq_{\cH} e\leq_{\cH} a$. We have to show that $ef$ is an idempotent. $efef=effe=efe=eef=ef$ and $ef$ is an idempotent.
\endproof

\section{The natural generalized inverse in a semigroup}

\subsection{Definition and first properties}

\begin{definition}
Let $S$ be a semigroup, $a\in S$. 
\ben
\item
Let $j=0,1,2$. The element $a$ is $j-$maximally invertible if the set $\Sigma_j(a)$ admits maximal elements for the natural partial order. Elements $b=a^{\parallel e}$ where $e$ is maximal are then called $j-$maximal generalized inverses of $a$.
\item If there exists a greatest element $M\in \Sigma_j(a)$, then we say that $a$ is $j-$naturally invertible, and $b=a^{\parallel M}$ is called the $j-$natural (generalized) inverse of $a$.  
\item Finally, if $a$ is $2-$naturally invertible, the element $aM=aba$ is called the core of $a$.
\een
\end{definition}

We will mainly deal with the $2-$natural inverse in the sequel, and we will also refer to it as \textit{the natural inverse}. As noted before, if $S$ is commutative or the idempotents central then the three notions coincide.\\
  
Recall that a semillatice is distributive if $e\vee f\leq x$ implies the existence of $e',f'$ such that $e\leq e',\; f\leq f'$ and $x = e' \vee f'$ .

\begin{proposition}
If the semillatice $\Sigma_2(a)$ is distributive, then any $2-$maximally invertible element is naturally invertible.
\end{proposition}

\proof
let $e$ be a maximal element of $\Sigma_2(a)$, $f\in \Sigma_2(a)$. Then $ef=fe\leq e$ and exists  $e',f'$ such that $e\leq e',\; f\leq f'$ and $e = e' f'$. By maximality, $e'=e$ and we get $e=ef'=f'e$. It follows that $e\leq f'$ hence $e=f'$ and $f\leq e$. $e$ is the greatest element in $\Sigma_2(a)$.
\endproof

The natural inverse generalizes the Drazin inverse \cite{Drazin58}.
\begin{theorem}
Assume $a$ is Drazin invertible with inverse $a^D$. Then $a$ is $1$ and $2-$naturally invertible with inverse $a^{\parallel M}=a^D$.
\end{theorem}

\proof
Let $a$ be Drazin invertible with index $n$ and inverse $a^D$. Then $e=aa^D=a^Da\in \Sigma_2(a)\subset \Sigma_1(a)$, and $a^D a^{n+1}= a^{n+1} a^D= a^n$.
Let $f\in \Sigma_1(a)$. Then $a^{\parallel f}$ satisfies $a^{\parallel f}af=f=a a^{\parallel f} f=faa^{\parallel f}=fa^{\parallel f}a.$ It follows that $f=f\left(a^{\parallel f}\right)^{n+1} a^{n+1}$. Then $fe=f\left(a^{\parallel f}\right)^{n+1} a^{n+1} a^D a=f$. Also $fe=ef$ ($e=aa^D=a^Da\in \Sigma_2(a)$) hence $f\leq e$.
\endproof

%Actually, the idempotent $e=aa^D=a^Da$ is also the greatest element of $\Sigma_1$, since in the preceeding proof we only use the commutativity of $f$ with $a$. 
\subsection{Examples}

\underline{\textsc{Many Maximal inverses}.}\\

Let $S$ be the semigroup generated by three elements $e,f,a$ subject to the conditions $e=e^2=ea=ae$, $f=f^2=fa=af$ and $ef=fe$. Then $S$ is commutative, $a$ is maximally invertible but not naturally invertible, with two maximal inverses $a\de=e$ and $a^{\parallel f}=f$.\\

We consider now a simple variant of the previous example. Let $S'$ be the semigroup generated by three elements $e,f,a$ subject to the conditions $e=e^2=ea=ae$, $f=f^2=fa=af$ and $ef=e,\; fe=f$. Then $\Sigma_1(a)=\{e,f\}$ but $\Sigma_2(a)$ is empty since $e$ and $f$ do not commute.\\

\underline{\textsc{Right hereditary semigroups with central idempotents (\cite{Fountain77})}.}\\

In this example we notably show that elements of a right hereditary semigroup with central idempotents are naturally invertible, and describe the set $\Sigma(a)$.\\

Let $S$ be a right p.p. (principal projective) semigroup with central idempotents as defined in \cite{Fountain77}. Then $E(S)$ is a semillatice (for the natural partial order). For any $e\in E(S)$, define $Y_e=\{x\in S, xe=x \text{ and } xs=xt\Rightarrow es=et\}$ (that is the $\cL^*-$class of $e$ for the extended Green's relation $\cL^*$ \cite{Fountain82}). Then $Y_e$ is a cancellative monoid (with unit $e$) and the structure theorem of Fountain says that $S$ is the semillatice of these disjoints monoids.\\

If $S$ is right semi-hereditary then it is right p.p. and incomparable principal right ideals are disjoints\cite{Dorofeeva72}. It follows notably that $E(S)$ is a chain (any two idempotents are comparable), and maximal invertibility implies natural invertibility.\\

Let now $a\in S$. If $a$ is regular, then $a$ is group invertible hence naturally invertible. We assume in the sequel that $a$ is not regular. By centrality of the idempotents, $\Sigma_0(a)=\Sigma_1(a)=\Sigma_2(a)=\{e\in E(S), e\leq_{\cH} a\}$. 
Let $a^0\in E(S)$ be the idempotent such that $a\in Y_{a^0}$. Since $aa^0=a^0a=a$, any $e\leq_{\cH} a$ satisfies $e\leq a^0$ for the $\cH$ order hence the natural partial order, and since $a$ is not regular, $e<a^0$.
Conversely, let $e<a^0$ and assume $S$ is semi-hereditary. From $ae=ea\in eS\cap aS$, $eS$ and $aS$ are comparable, and from $e<a^0$ we get $eS\subset aS$ ($ae\in Y_e$ disjoint from $Y_{a^0}$ hence $ae\neq a$). It follows that $e\leq_{\cR} a$ and in particular $ea=ae\cR e$ is regular. Since for regular elements, the appartenance in $Y_e$ is simply Green's relation $\cL$, we get that $ae=ea\cH e$ and $e\leq_{\cH} a$. Finally, we have proved that $\Sigma(a)$ is the chain of idempotents $\{e\in E(S), e<a_0\}$. \\

If we finally assume that $S$ is right hereditary (right ideals are projective), then Dorofeeva \cite{Dorofeeva72} showed that $S$ satifisies the maximum condition for principal right ideals. As a consequence, the chain of idempotents $\{e\in E(S), e<a_0\}$ has a greatest element $M$ and $a$ is naturally invertible with inverse $a^{\parallel M}$.

\section{The ring case}

\subsection{Invertibility along an element in a ring}
In \cite{maryPatricio}, invertibilty along an element was characterized in terms of existence of units.

\begin{theorem}\label{thunit}
Let $d$ be a regular element of a ring $R$, $d^-\in A(d)$. Then the following are equivalent:
\ben
\item $a\dd$ exists.
\item $u=da+1-dd^-$ is a unit.
\item $v=ad+1-d^-d$ is a unit.
\een
In this case,
$$a\dd= u\io d = d v\io.$$
\end{theorem}

Note that in the particular case of invertibility along an idempotent $e$, thise reduces to:
\begin{corollary}
Let $e\in E(R)$ be a idempotent element of a ring $R$. Then the following are equivalent:
\ben
\item $a\de$ exists.
\item $u=ea+1-e$ is a unit.
\item $v=ae+1-e$ is a unit.
\een
In this case,
$$a\de= u\io e = e v\io.$$
\end{corollary}

\begin{corollary}
If $ae=ea$, then $e\leq_{\cH} a$ if and only if $u=1+ae-e$ is a unit.
\end{corollary}

Remark that a sufficient condition for this to happen is the following:

\begin{lemma}\label{lemmaK}
If $ae=ea$ and $a+1-e$ is a unit, then $e\leq_{\cH} a$.
\end{lemma}

\proof
let $u=a+1-e$. Then $ue=ae=ea$ hence $e= u\io e a=a u\io e$.
\endproof

\subsection{Natural inverse in a ring}

Let $R$ be a ring, and let $a\in R$. Then the semillatice $\Sigma_2(a)$ is actually a distributive lattice (with $e\wedge f=e+f-ef$) hence $a$ is $2-$maximally invertible if and only if it is naturally invertible.

We derive new criterion for the natural inverse to exists.
\begin{theorem}
Let $a\in R$. Then the following are equivalent:
\ben
\item $a$ is naturally invertible with inverse $a^{\parallel M}$;
\item There exists $b\in \{a\}''$, $bab=b$ and $\Sigma_2(a-aba)=\{0\}$;
\item $a=x+y$ with $x\in \{a\}''$, $x\grp$ exists, $xy=0$ and $\Sigma_2(y)=\{0\}$.
\een
In this case, $a^{\parallel M}=b=x\grp$. 
\end{theorem}

\proof
$\,$\\
\vspace{-.5cm}
\ben
\item[$1)\Rightarrow 2)$] Assume $a$ is naturally invertible with inverse $b=a^{\parallel M}$. Then $M=ab=ba$. Let $e\in \Sigma_2(a-aba)$. The $ca=ac\Rightarrow cb=bc\Rightarrow c(a-aba)=(a-aba)c\Rightarrow ec=ce$. Hence $e\in \{a\}''$. But also $\exists t,s\in R,\; e=a(1-ba)t=s(1-ab)a$ and $e\leq_{\cH} a$. Finally, $e\in \Sigma_2(a)$ hence $e\leq M,\; eM=Me=e$.
Computation give $e=eM=s(1-ab)aba=0$. 
\item[$2)\Rightarrow 3)$] Let $b\in \{a\}''$, $bab=b$ and $\Sigma(a-aba)=\{0\}$. Then $x=aba$ and $y=a-aba$ satisfy the required relations ($x\grp =b$).
\item[$3)\Rightarrow 1)$] Finally, let $a=x+y$ with $x\in \{a\}''$, $x\grp$ exists, $xy=0$ and $\Sigma(y)=\{0\}$. By properties of the group inverse, $x\grp\in \{x\}''\Rightarrow xx\grp\in \{a\}''$. Pose $M=xx\grp$. Since $x=xx\grp x=xx\grp a=axx\grp$, $M\leq_{\cH} a$ and $M\in \Sigma_2(a)$. Let $e\in \Sigma_2(a)$. Then $e=a\de a e$, and $e$ is in the bicommutant of $y=a-x$. Then $$e-eM=e(1-xx\grp)=e a\de a(1-xx\grp)=e a\de (x+y)(1-xx\grp)=e a\de y$$ and $e-eM\in \Sigma_2(y)$. By hypothesis, $e-eM=0$ and $e\leq M$, $M$ is the greatest element of $\Sigma_2(a)$ and $a$ is naturally invertible with inverse $a^{\parallel xx\grp}$. 
\een
\endproof

The unique decomposition $a=x+y=aM+(a-aM)=aba+(a-aba)$ as in the previous theorem will be called the natural core decomposition of $a$. 

\subsection{Link with the Koliha-Drazin inverse}

We recall the following definitions of quasinilpotency and quasipolarity in rings due to R. Harte \cite{Harte91}. 

\begin{definition}
An element $q$ of a ring $R$ is quasinilpotent if $\forall x\in \{q\}', 1+xq\in R^{-1}$, and quasi-quasinilpotent if $\forall x\in \{q\}'', 1+xq\in R^{-1}$
\end{definition}

Note that quasi-quasinilpotent elements need not be quasinilpotent in general. The two notions however coincide for Banach algebras.

\begin{definition}
An element $a$ of a ring $R$ is quasipolar (resp. quasi-quasipolar) if there exists a idempotent (called spectral idempotent) $p$ in $\{a\}''$ such that $ap$ is quasinilpotent (resp. quasi-quasinilpotent) and $a+p\in R^{-1}$.
\end{definition}

It was remarked in \cite{Koliha02} that the last condition can be replaced by the following one $1-p\leq_{\cH} a$. This is the content of lemma \ref{lemmaK}.

It was proved by J. Koliha and P. Patricio (theorem 4.2 in \cite{Koliha02}) that quasipolar elements are exactly the generalized Drazin invertible elements (also called Koliha-Drazin invertible elements):
\begin{definition}
An element $a$ of a ring $R$ is generalized Drazin invertible if there exists $b$ in $\{a\}''$ such that $bab=b$  and $a^2b-a$ is quasinilpotent.
\end{definition}

\begin{theorem}
An element $a$ of a ring $R$ is generalized Drazin invertible if and only if it is quasipolar. In this case $b=(a+p)^{-1}(1-p)$.
\end{theorem}

Next theorem proves that the natural inverse generalizes not only the Drazin inverse, but also the Koliha-Drazin inverse in a ring:

\begin{theorem}\label{thring}
Let $R$ be a ring, and $a\in R$ be quasi-quasipolar with spectral idempotent $p$ and Koliha-Drazin inverse $b$. Then $a$ is naturally invertible, $M=1-p$ is the greatest element of $\Sigma_2(a)$ and the generalized Drazin inverse $b$ is equal to $a^{\parallel M}$, the natural generalized inverse of $a$.
\end{theorem}

\proof
%Assume $a$ is quasipolar in the semigroup sense and let $b=a\de$. We have to prove that $q=a^2b-a$ is quasinilpotent. Let $x\in \{q\}'$.
Assume $a$ is quasi-quasipolar in the ring sense. Then exists $p$ spectral idempotent, $p$ in $\{a\}''$ such that $ap$ is quasi-quasinilpotent and $a+p\in R^{-1}$. By lemma \ref{lemmaK}, $M=1-p\leq_{\cH} a$, hence it is in $\Sigma_2(a)$. Let $f\in \Sigma_2(a)$. Then exists $x\in S, f=xa$. By quasi-quasinilpotency, $(1-fp)=(1-xap)\in R^{-1}$. But by commutativity of $\{a\}''$ and the fact that $f,p\in E(S)$, we have $(1-fp)(1+fp)=1-fp$. By invertibility, $1+fp=1$ hence $fp=0$. It follows that $fM=f(1-p)=f-fp=f$ and $f\leq M$ for the natural partial order. $M$ is the greatest element of $\Sigma_2(a)$. 
Now the generalized Drazin inverse of $a$ $b=(a+p)^{-1} (1-p)$ is obviously in $\cH_{(1-p)}$ and is an outer inverse of $a$ by definition. By unicity, it is $a^{\parallel M}$.
\endproof

If we require the element $a$ to be quasipolar instead of quasi-quasipolar with Koliha-Drazin inverse $b$, then the idempotent $M=1-p$ is actually the greatest element $\Sigma_1(a)$ and $b=a^{\parallel M}$ is the $1-$natural generalized inverse of $a$. It is not in $\{a\}''$ in general. 

\section{The Banach Algebra case}

%\subsection{Banach Algebra}
If $a\in A$ with $A$ a Banach algebra, we denote the spectrum of $a$ by $\sigma(a)$ and the spectral radius by $r(a)$.\\

Recall that in a Banach algebra, an element is quasinilpotent if its spectrum reduces to $0$, or equivalently if its spectral radius is $0$, and quasipolar if $0$ an isolated point of the spectrum. It is known \cite{Harte91} that these notions coincide with their ring counterpart, and also with the quasi-quasi notion (for instance, $\sigma(a)=\{0_{\CC}\}$ if and only if  $a$ is quasinilpotent in the ring sense if and only if $a$ is quasi-quasinilpotent in the ring sense).

\begin{corollary}
Let $a\in A$. If $0$ is an isolated point of the spectrum of $a$, then $a$ is naturally invertible.
\end{corollary}

\proof
If $0$ is an isolated point of the spectrum of $a$, then $a$ is quasi-quasipolar in the Banach sense, hence it is quasi-quasipolar ring sense. We then apply theorem \ref{thring}.
\endproof

We now investigate the link between $\Sigma_i(a), i=1,2$ and $\sigma(a)$.

\begin{theorem}
Let $A$ be a unital Banach algebra, $a\in A$. Then 
\begin{enumerate}
\item $\sigma(a)=\{0_{\CC}\}\Rightarrow \Sigma_1(a)=\{0\}$.
\item $\Sigma_2(a)=\{0\}\Rightarrow \sigma(a)$ is connected and contains $0_{\CC}$.
\end{enumerate}
\end{theorem}

\begin{proof}
\begin{enumerate}
\item If the spectrum of $a$ reduces to $0$, the its spectral radius is equal to $0$. Let $e\in \Sigma_1(a)$. Then $e=aa\de=a\de a$.
We get $$||e||^{\frac{1}{n}}=||e^n||^{\frac{1}{n}}=||a^n (a\de)^n||^{\frac{1}{n}}\leq ||a^n||^{\frac{1}{n}} ||(a\de)||\to 0$$
and $||e||=0$.
\item 
If $a$ is invertible, then $1\in \Sigma_2(a)$. Hence assume $\sigma(a)$ contains $0$ but is not connected. Then $\sigma(a)=C_0\cup C_1$ with $0\in C_0$ and $C_0, C_1$ disjoint and open and closed in $\sigma(a)$. Then the holomorphic calculus for $f(z)=\frac{1}{z}$ on $U$ open set containg $C_1$ and $0$ outside $U$, such that $U$ contains an open neighbourhood of $C_0$, defines an element $x=f(a)$ of $\{a\}''$ such that $ax=xa=e$ is idempotent and non zero, and $\Sigma_2(a)$ does not reduce to $\{0\}$.
\end{enumerate}
\end{proof}

Now, we consider three different (commutative) Banach algebras to show that we cannot do better in the theorem, nor define natural invertibility in terms of the spectrum. 

\begin{itemize}
\item Consider the Banach algebra $C_{0}([0,1])$ of continuous functions on $[0,1]$, and let $a(t)=t$. Then $\sigma(a)=[0,1]$ and $\Sigma(a)=\{0\}$. $a$ is naturally invertible with $b=0$.
\item Consider the Banach algebra $C_{0}([0,1]\cup[2,3])$ of continuous functions on $[0,1]\cup [2,3]$, and let $a(t)=t, 0\leq t\leq 1$ and $a(t)=t-1, 2\leq t\leq 3$. Then $\sigma(a)=[0,2]$ and $\Sigma(a)=\{\mathbf{1}_{[2,3]}\}$. $a$ is naturally invertible with $b(t)=0, 0\leq t\leq 1$ and $b(t)=\frac{1}{t-1}, 2\leq t\leq 3$.
\item Consider now the Banach algebra $L^{\infty}([0,1])$ of essentially bounded measurable functions  on $[0,1]$, and let $a(t)=t$. Then $\sigma(a)=[0,1]$ and $\Sigma(a)=\{\textbf{1}_A,\; \exists 0<c\leq 1, \lambda(A\cap [0,c])=0\}$. This set admits no maximal element, hence $a(t)=t$ is not naturally invertible. 
\end{itemize}

It appears that natural invertibility is strongly linked with the nature of the structure space (or spectrum) of the whole commutative Banach algebra $B=\{a\}''$, independently of the nature of the spectrum of the element $a$. Obviously, if the spectrum of $\{a\}''$ is not connected, then Shilov's idempotent theorem gives the existence of a nontrivial idempotent. This idempotent needs not to be in $\Sigma(a)$.\\

Next theorem uses the generalized spectral theory of Hile and Pfaffenberger (\cite{Hile85}, \cite{Hile87}) and the associated functional calculus to construct elements in $\Sigma_j(a), j=1,2$. The construction is similar to the case of a disconnected spectrum, but instead of using $\sigma(a)$ (that can be connected), we use the generalized spectrum of Hile and Pfaffenberger. If $a,q\in A$, then the spectrum of $a$ relative to $q$, or $q-$spectrum of $a$ $\sigma_q(a)$, is the set of points $z$ such that $a-z.1-\bar{z}q$ is not invertible in $A$.

\begin{theorem}
Let $a,q\in A$, with $\sigma(a)$ connected set that contains $0$. Assume $\sigma(q)\cap \mathbb{T}=\emptyset$, where $\mathbb{T}$ is the unit circle, and $\sigma_q(a)$ is not connected. Then
\ben
\item if $q\in \{a\}'$, $\Sigma_1(a)$ is not empty;
\item if $q\in \{a\}''$, $\Sigma_2(a)$ is not empty.
\een
\end{theorem} 

\proof
This is a consequence of Theorem 12 in \cite{Hile85}. Indeed, since $a$ is not invertible, $0$ is in the $q$ spectrum of $a$. Since $\sigma_q(a)$ is not connected, we can find closed rectifiable curve $\Gamma$ in the $q$ resolvent such that $0$ is in its exterior and its interior contains elements of $\sigma_q(a)$ (a component of $\sigma_q(a)$ that does not contains $0$).
Choosing $z=0$ in equation $4.3$ gives an idempotent $p\leq_{\cR} a$. The rest follows from commutation properties.
\endproof

%Obviously, such a $q$ need not exist in general. However, Hile and Paffenberger give a example of $a$ and $q$ with connected spectrum such that $\sigma_q(a)$ is not connected.
%We could also use results of M. Berkani \cite{Berkani96}. 
\section{Operators}

\subsection{Local spectral theory}
%\underline{\textsc{Local spectral theory}}
In the operator case, we can improve somehow the results of the previous section. Let $X$ be a Banach space and $T\in \cB(X)$. $T(X)$, or $R(T)$ denotes its range, $N(T)$ its kernel. We use ideas from local spectral theory (\cite{Mbekhta87}, \cite{Mbekhta08}, \cite{Aiena04}, \cite{Harte09local}) and define the following sets:
\begin{definition}
$\,$\\
\vspace{-.5cm}
\begin{itemize}
\item The hyperrange of $T$ is the linear space $T^{\infty}(X)=\bigcap_{n\in \NN} T^n (X)$;
\item The hyperkernel of $T$ is the linear space $N^{\infty}(T)=\bigcup_{n\in \NN} N(T^n)$;
\item The quasinilpotent part (or transfinite kernel) of $T$ is the linear space $H_0(T)=\{x\in X,\; ||T^n x||^{\frac{1}{n}}\to 0\}$;
\item The algebraic core of $T$ $C(T)$ is the largest subspace such that $T(M)=M$;
\item The analytic core (or transfinite range) of $T$ $K(T)$ consists of all vectors $x_0 \in X$ for which there exist a sequence
$x_n \in X$ such that $Tx_n = x_{n-1}$ and exists $c>0, ||x_n||\leq c^n||x_0||$.
\end{itemize}
\end{definition}

The algebraic core can also be defined as follows : $C(T)$ consists of all vectors $x_0 \in X$ for which there exist a sequence
$x_n \in X$ such that $Tx_n = x_{n-1}$. We then have the following inclusions :
$$K(T)\subset C(T) \subset T^{\infty}(X),\quad N^{\infty}(T)\subset H_0(T).$$

In \cite{Harte09local}, it is proved that for a bounded operator $T$, the analytic core corresponds to the holomorphic range $\{\lim_{z\to 0} (T-zI)f(z),\; f\in Holo(0,X)\}$, and that the intersection if the analytic core with $N(T)$ is the holomorphic kernel of $T$ $\{g(0),\; (T-zI)g(z)=0,\; g\in Holo(0,X)\}$.

We have the following relations:
\begin{proposition}\label{propincl}
Let $P\in \Sigma_1(T)$. Then $P(X)\subset K(T)$ and $H_0(T)\subset N(P)$.
\end{proposition}

\proof
Let $P\in \Sigma_1(T)$. Then $P=TT^{\parallel P}P= TT^{\parallel P}=T^{\parallel P}T$.\\
Let $x_0\in P(X)$, and for all $n>0$, pose $x_n=(T^{\parallel P})^n x_0$. Then $Tx_{n}=T(T^{\parallel P})^n x_0=P(T^{\parallel P})^{n-1} x_0=(T^{\parallel P})^{n-1} Px_0=(T^{\parallel P})^{n-1} x_0=x_{n-1}$. Also $||x_n||\leq ||T^{\parallel P}||^n ||x_0||$, hence $x_0\in K(T)$.\\
Let now $x\in H_0(T)$. Then $P(x)=T^{\parallel P}T(x)= (T^{\parallel P})^n T^n(x)$ forall $n>0$ and $$||P(x)||^{\frac{1}{n}}\leq ||(T^{\parallel P})^n||^{\frac{1}{n}} ||T^n(x)||^{\frac{1}{n}}\leq ||(T^{\parallel P})|| ||T^n(x)||^{\frac{1}{n}}\to 0$$
and $P(x)=0$.
\endproof

\begin{corollary}
$$K(T)=\{0\}\Rightarrow \Sigma_1(T)=\{0\};\qquad \overline{H_0(T)}=X\Rightarrow \Sigma_1(T)=\{0\}.$$
\end{corollary}

Obviously, the existence of a greatest element in $\Sigma_2(T)$ is guaranted by a decomposition of the form $X=H_0(T)\oplus K(T)$, with both subspaces closed (choose $P$ the associted projection on $K(T)$). But such a decomposition occurs only for quasipolar elements:

\begin{theorem}[\cite{Mbekhta87}, theorem 1.6]
Let $T\in \cB(X)$. Then $0$ is an isolated point of the spectrum if and only if $H_0(T)$, $K(T)$ are closed and $X=H_0(T)\oplus K(T)$.
\end{theorem}

\begin{theorem}
Assume $K(T)$ is closed and hyperinvariantly complemented, with complement $N$ and $N(T)\cap K(T)=\{0\}$.
Then $T$ is naturally invertible with greatest idempotent the projection on $K(T)$ parallel to $N$.
\end{theorem}

\proof
Let $X=K(T)\oplus N$ and $M$ the idempotent of the theorem. First, we must prove that $M\in \Sigma_2(T)$. Since $K(T)$ and $N$ are hyperinvariant, we only have to prove that $M\leq_{\cH} T$. Consider $T_{|K(T)}:K(T)\to K(T)$ the restriction of $T$ to $K(T)$. $T_{|K(T)}$ is well defined since $T(K(T))\subset K(T)$, and surjective since $T(K(T))= K(T)$. But from the hypothesis $N(T)\subset N$ it is also injective, hence invertible and exists $S$ bounded operator, $TS=ST=M$.\\
Let now $P$ be and idempotent in $\Sigma_2(T)$. Then $P(X)\subset K(T)$ from proposition \ref{propincl}, hence $P(X)\subset M(X)$. It follows that $PMP=P$ and by commutation ($\Sigma_2(T)$ is a commutative semigroup), $PM=MP=PMP=P$ and $M$ is the greatest element of $\Sigma_2(T)$.
\endproof

By the results of Harte \cite{Harte09local}, $N(T)\cap K(T)$ is the holomorphic kernel of $T$, and it reduces to $0$ precisely when $T$ has the single valued extension property (SVEP) at $0$ (Theorem 9 p.~180).

As a final result, we investigate the range of the core of a naturally invertible element:

\begin{proposition}
Let $T$ be naturally invertible with natural inverse $B$, greatest idempotent $M=TB=BT$ and core $TM=TBT$. Then
$K_{\nu}(T)=TM(X)$ is a closed, hyperinvariant, complemented (with hyperinvariant complement) subspace of the analytic core $K(T)$, and $TK_{\nu}(T)=K_{\nu}(T)$.
\end{proposition}

\proof
By commutation, $TM(X)\subset M(X)$. But also $M(X)=M^2(X)=BTM(X)\subset TM(X)$ and the two subspaces are equal. The other properties follow.  
\endproof

\subsection{Miscellanous}
In this last section we give examples and results relative to natural invertiblity.

\underline{\textsc{The shift operator}}

Let $S$ be the shift operator on $l^2(\NN)$. Then $S$ is not quasinilpotent, but its hyperrange reduces to $0$. As a consequence, $\Sigma_1(S)=\{0\}$. The spectrum of $S$ is the unit disk.

\underline{\textsc{Strongly irreducible operators}}

In 1972, F. Gilfeather\cite{Gilfeather72} introduced the concept of strongly
irreducible operator. A bounded linear operator $T$ is said to
be strongly irreducible, if there exists no non-trivial idempotent
$p$ in the commutant of $T$. This concept actually coincide with the concept of 
Banach irreducible operator (a bounded linear operator $T$ is said
to be Banach irreducible, if $T$ can not be written as a direct
sum of two bounded linear operators). It is clear that strongly irreducible operators satisfy $\Sigma_1(T)=\{0\}$.

Also, the following spectral result is due to Herrero and Jiang \cite{Herrero90}:
\begin{theorem}
$\sigma(T)$ is connected if and only if $T$ is in the norm closure of strongly irreducible operators.
\end{theorem}

\underline{\textsc{Rosenblum's corollary, commutant and bicommutant}}

Let $T=\left(\begin{array}{cc}
X & 0   \\
0 & Y \end{array} \right)$ be the a decomposition of $T$ with $X$ invertible and $M=\left(\begin{array}{cc}
XX^{-1} & 0   \\
0 & 0 \end{array} \right)$ the greatest element of $\Sigma_1(T)$.  
If $\sigma(X)\cap \sigma(Y)=\{0\}$, then by Rosenblum's corollary (see \cite{Radjavi73}), $\left(\begin{array}{cc}
X & 0   \\
0 & 0 \end{array} \right),\left(\begin{array}{cc}
0 & 0   \\
0 & Y \end{array} \right)\in \{T\}''$ and $T=\left(\begin{array}{cc}
X & 0   \\
0 & 0 \end{array} \right)+\left(\begin{array}{cc}
0 & 0   \\
0 & Y \end{array} \right)$ is the natural core decomposition of $T$, with $M=\left(\begin{array}{cc}
XX^{-1} & 0   \\
0 & 0 \end{array} \right)$ the greatest element of $\Sigma_2(T)$. This is the case for instance when $Y$ is quasinilpotent.\\

\bibliographystyle{elsarticle-num}

\end{document}